\documentstyle[12pt,twoside]{article}
\def\Cchi{{\raisebox{.2ex}{ \large $\chi$}}}
\newtheorem{prethm}{{\bf Theorem}}

\newenvironment{thm}{\begin{prethm}{\hspace{-0.5
em}{\bf.}}}{\end{prethm}}
\newtheorem{precor}{{\bf Corollary}}

\newtheorem{preprop}{{\bf Proposition}}

\newenvironment{prop}{\begin{preprop}{\hspace{-0.85
em}{\bf.}}}{\end{preprop}}
\newtheorem{preque}{{\bf Question}}

\newtheorem{preques}{{\bf Question}}

\newtheorem{prelemma}{{\bf Lemma}}

\newenvironment{lemma}{\begin{prelemma}{\hspace{-0.5
em}{\bf.}}}{\end{prelemma}}
\newtheorem{prelemm}{{\bf Lemma}}

\newtheorem{preex}{{\bf Example}}

\newtheorem{prepro}{{\bf Proposition}}

\newtheorem{preobs}{{\bf Observation}}

\newenvironment{obs}{\begin{preobs}{\hspace{-0.5
em}{\bf.}}}{\end{preobs}}
\newtheorem{preprob}{{\bf Problem}}

\newtheorem{prelem}{{\bf Theorem}}

\newtheorem{preproof}{{\bf Proof.}}

\newenvironment{proof}[1]{\begin{preproof}{\rm
               #1}\hfill{$\Box$}}{\end{preproof}}

\newtheorem{preconj}{{\bf Conjecture}}

\newenvironment{conj}{\begin{preconj}{\hspace{-0.5
em}{\bf.}}}{\end{preconj}}
\newtheorem{predeff}{{\bf Definition}}

\newenvironment{deff}{\begin{predeff}{\hspace{-0.85
em}{\bf.}}}{\end{predeff}}

\input{amssym}

\date{}
\title{{\large\bf On the Locating Chromatic Number of the Cartesian Product of Graphs}}

{\small
\author{
{\sc Ali Behtoei\footnote{{\small \it alibehtoei@math.iut.ac.ir}}}
~and
{\sc Behnaz Omoomi\footnote{{\small \it  bomoomi@cc.iut.ac.ir}}}\\
[1mm]
{\small \it  Department of Mathematical Sciences}\\
{\small \it  Isfahan University of Technology} \\
{\small \it 84156-83111, \ Isfahan, Iran}}

\begin{document}
\small \voffset=-20mm

\maketitle

\begin{abstract}
 Let $c$ be a proper $k$-coloring of a connected graph $G$  and
$\Pi=(V_1,V_2,\ldots,V_k)$ be an ordered partition of $V(G)$ into
the resulting color classes. For a vertex $v$ of $G$, the color
code of $v$ with respect to $\Pi$ is defined to be the ordered
$k$-tuple $c_{{}_\Pi}(v):=(d(v,V_1),d(v,V_2),\ldots,d(v,V_k)),$
where $d(v,V_i)=\min\{d(v,x)~|~x\in V_i\}, 1\leq i\leq k$. If
distinct vertices have distinct color codes, then $c$ is called a
locating coloring.  The minimum number of colors needed in a
locating coloring of $G$ is the locating chromatic number of $G$,
denoted by  $\Cchi_{{}_L}(G)$. In this paper, we study the
locating chromatic numbers of grids, the cartesian product of paths
and complete graphs, and the cartesian product of two complete
graphs.
\end{abstract}


\section{Introduction}
 Let $G$ be a graph
without loops and multiple edges with vertex set $V(G)$ and edge
set $E(G)$. A proper $k$-coloring of $G$, $k\in \Bbb{N}$,  is a
function $c$ defined from $V(G)$ onto a set of colors
$[k]:=\{1,2,\ldots,k\}$ such that every two adjacent vertices have
different colors. In fact, for every $i$, $1\leq i\leq k$, the set
$c^{-1}(i)$ is a nonempty independent set of vertices which is
called the color class $i$. The minimum cardinality $k$ for which
$G$ has a proper $k$-coloring is the chromatic number of $G$,
denoted by $\Cchi(G)$. For a connected graph $G$, the distance
$d(u,v)$ between two vertices $u$ and $v$ in $G$ is the length of
a shortest path between them, and for a subset $S$ of $V(G)$, the
distance between $u$ and $S$ is given by
$d(u,S):=\min\{d(u,x)~|~x\in S\}$.

\begin{deff} {\rm \cite{XL}}
Let $c$ be a proper $k$-coloring of a connected graph $G$ and
$\Pi=(V_1,V_2,\ldots,V_k)$ be an ordered partition of $V(G)$ into
the resulting color classes. For a vertex $v$ of $G$, the color
code of $v$ with respect to $\Pi$ is defined to be the ordered
$k$-tuple $$c_{{}_\Pi}(v):=(d(v,V_1),d(v,V_2),\ldots,d(v,V_k)).$$
  If distinct vertices of $G$ have distinct color codes, then $c$ is
called a locating coloring of $G$. The locating chromatic number,
$\Cchi_{{}_L}(G)$,  is the minimum number of colors in a locating
coloring of $G$.
\end{deff}

The concept of locating coloring was first introduced and studied
by Chartrand et al. in \cite{XL}. They established some bounds for
the locating chromatic number of a connected graph. They also
proved that for a connected graph $G$ with $n\geq 3$ vertices, we
have $\Cchi_{_L}(G)=n$ if and only if $G$ is a complete
multipartite graph. Hence, the locating chromatic number of the
complete graph $K_n$ is $n$. Also for paths and cycles of order
$n\geq 3$ it is proved in \cite{XL} that $\Cchi_{_L}(P_n)=3$,
$\Cchi_{_L}(C_n)=3$ when $n$ is odd, and $\Cchi_{_L}(C_n)=4$ when
$n$ is even. The locating chromatic number of trees, Kneser
graphs, and the amalgamation of stars are studied in \cite{XL},
\cite{Behtoei}, and \cite{Asmiati}, respectively. For more
results in the subject and related subjects, see~\cite{Asmiati}
to \cite{Conditional}.

Obviously, $\Cchi(G)\leq \Cchi_{{}_L}(G)$. Note that the $i$-th
component of the color code of each vertex in the color class
$V_i$ is zero and its other components are non zero. Hence, a
proper coloring is a locating coloring whenever the color codes of vertices
in each color class are different. In a proper coloring of $G$, a
vertex is called {\it colorful} if all of the colors appear in its
closed neighborhood, and the color of a colorful vertex is called
a {\it full} color. Note that in each proper $m$-coloring of $K_m$ all of the vertices are colorful. We have the following observation.

\begin{obs}
Let $G$ be a connected graph. (a) In a locating coloring of $G$, there are no two colorful vertices that are assigned the same color. Therefore, if there is a locating $k$-coloring of $G$, then there are at most $k$ colorful vertices. (b) If $G$ contains two disjoint cliques of order $k$, then $\Cchi_{{}_L}(G)\geq k+1$.
\end{obs}


Recall that the cartesian product of two graphs $G$ and $H$,
denoted by $G\square H$, is a graph with vertex set $V(G)\times
V(H)$ in which two vertices $(a,b)$ and $(a',b')$ are adjacent in
it whenever $a=a'$ and $bb'\in E(H)$, or $aa'\in E(G)$ and $b=b'$.
Vertices of the cartesian product $G\square H$ can be represented
by an $|V(G)|$ by $|V(H)|$ array, such that the induced subgraph
on the vertices of each row is isomorphic to $H$ and the induced
subgraph on the vertices of each column is isomorphic to $G$. In
this paper, we study the locating chromatic number of the grid
$P_m\square P_n$, $K_m\square P_n$ and $K_m\square K_n$.

\section{The locating chromatic numbers of $P_m\square P_n$ and $K_m\square P_n$}

In this section, we determine the exact value of the locating
chromatic number of the grid $P_m\square P_n$ and  $K_m\square
P_n$. First, we give an upper bound for the locating chromatic
number of the cartesian product of two arbitrary connected graphs.

\begin{prop} \label{upperbound}
If $G$ and $H$ are two connected graphs, then $\Cchi_{_L}(G\square H)\leq\Cchi_{_L}(G)\Cchi_{_L}(H)$.
\end{prop}
\begin{proof}{
Let $m:=\Cchi_{_L}(G)$ and $A_1,A_2,...,A_m$ be the color classes
of a locating $m$-coloring of $G$. Also, let $n:=\Cchi_{_L}(H)$ and
$B_1,B_2,...,B_n$ be the color classes of a locating $n$-coloring of
$H$. For each $i\in[m]$ and each $j\in [n]$, $A_i\times B_j$ is an
independent set in $G\square H$. Hence, the partition $\{
A_i\times B_j~|~i\in[m],~j\in[n]\}$ of vertices of $G\square H$
can be considered as the color classes of a proper coloring of
$G\square H$. To see that this is a locating coloring, let $(a,b)$
and $(a',b')$ be two distinct vertices in the color class
$A_i\times B_j$ and, without loss of generality, assume that
$a\neq a'$. Note that $d(b,B_j)=d(b',B_j)=0$ while, by assumption, there exists $k\in [m]\setminus \{i\}$ such that
$d(a,A_k)\neq d(a',A_k)$. Hence
\begin{eqnarray*}
d((a,b),A_k\times B_j)&=& d(a,A_k)+d(b,B_j)\\
&=& d(a,A_k)+0 \\
&\neq & d(a',A_k)+0\\
&=& d((a',b'),A_k\times B_j).
\end{eqnarray*}
Thus, this coloring is a locating coloring.
 }\end{proof}
For $G=H=K_2$, we have
$$\Cchi_{_L}(K_2\square K_2)=\Cchi_{_L}(C_4)=4=\Cchi_{_L}(K_2)\Cchi_{_L}(K_2).$$
Therefore, the above inequality is attainable. The following theorem shows that
the exact value of the locating chromatic number of an $m$ by $n$
grid $P_m\square P_n$ is $4$, while the given upper bound is $9$.
\begin{thm}
If $n\geq m\geq 2$, then $\Cchi_{_L}(P_m\square P_n)=4$.
\end{thm}
\begin{proof}{
In each proper $3$-coloring of $P_m\square P_n$ there exists an
induced cycle $C_4$ with $3$ colors. Hence, there are two colorful
vertices on this cycle with the same color. Therefore,
$\Cchi_{_L}(P_m\square P_n)\geq 4$.\\ For each $i\in[m]$ and
$j\in[n]$, let $v_{i,j}$ be the vertex in the $i$-th row and
$j$-th column of the grid $P_m\square P_n$, and let $c$ be a
proper $2$-coloring of  the bipartite graph $P_m\square P_n$ with
the color set $\{1,2\}$. Define the coloring $c'$ as $c'(v_{1,1})=3$,
$c'(v_{1,n})=4$ and $c'(v_{i,j})=c(v_{i,j})$ else where. For each
$i\in[m]$ and $j\in[n]$, we have
$$d(v_{i,j},v_{1,1})=i+j-2, ~~d(v_{i,j},v_{1,n})=n+i-j-1.$$
Thus, distinct vertices have distinct color codes with respect to
the coloring $c'$.
 }\end{proof}
Let $G:=K_m\square P_n$. Vertices of $G$ can be represented by an
$m$ by $n$ array. Thus, $G$ consists of $m$ rows and $n$ columns,
in which the induced subgraph on the vertices of each column is
isomorphic to $K_m$ and the induced subgraph on the vertices of
each row is isomorphic to $P_n$. Let $v_{i,j}$ be the vertex of
$G$ in the $i$-th row and $j$-th column. Hence, each coloring of
$G$ can be represented by an $m\times n$ matrix, in which its
$(i,j)$-entry is the color of $v_{i,j}$. For the locating
chromatic number of $K_m\square P_n$, the following cases are
easy to check (see Theorem 1 and \cite{XL}).
   \begin{itemize}
   \item [(a)] $\Cchi_{_L}(K_1\square P_1)=1$, $\Cchi_{_L}(K_1\square P_2)=2$, and $\Cchi_{_L}(K_1\square
   P_n)=3$, $n\geq 3$.
   \item [(b)] $\Cchi_{_L}(K_2\square P_n)=\Cchi_{_L}(P_2\square
   P_n)=4$.
   \item [(c)] $\Cchi_{_L}(K_m\square P_1)=\Cchi_{_L}(K_m)=m$.
   \end{itemize}
In the following theorem, the exact value of $\Cchi_{_L}(K_m\square
P_n)$ is computed in the remaining general case.
\begin{lemma} \label{ConsecutiveColumns}
Let $m\geq 3$ and $n\geq 2$ be two positive integers. If there exists a locating $(m+1)$-coloring of $G:=K_m\square P_n$, then let $C$ be its coloring matrix. Then every two consecutive columns of $C$ have
different missing colors. Moreover, if $m\geq 5$, then every two
columns of $C$ have different missing colors.
\end{lemma}
\begin{proof}{
First, let $m=3$. Suppose on the contrary, there exist two
consecutive columns $C_j$ and $C_{j+1}$ of $C$ with the same
missing color, say $``4"$. Assume that $C_j=[1~2~3]^T$. The
coloring is proper and hence $C_{j+1}$ is a derangement of $C_j$,
which implies $C_{j+1}=[3~1~2]^T$ or $C_{j+1}=[2~3~1]^T$. Without
loss of generality, assume that $C_{j+1}=[3~1~2]^T$. If $j=1$ and
$C_{j'}$ is the first column which the color $4$ appears in it,
say in row $i$, then two vertices $v_{i,1}$ and $v_{i+1,2}$ ($i+1$
considered modulo $m$) have the same distance to the color class
$4$ and hence, have the same color code, which is a contradiction.
For $j+1=n$, the argument is similar. Also, when all of the
columns containing the color $4$ have index greeter than $j+1$, or
all have index smaller than $j$, the argument is similar to above.
Thus, assume there exist two indices $j_1$ and $j_2$,
$j_1<j<j+1<j_2$, such that the color $4$ appears in both of the
columns $C_{j_1}$ and $C_{j_2}$, and does not appear in the columns
with indices $j_1<k<j_2$. If $j-j_1=j_2-(j+1)$, then exactly four
vertices in the $j$-th and $(j+1)$-th columns of $G$ have the
same distance $j-j_1+1$ to the color class $4$, and hence, at
least two vertices of the same color have the same color code.
Thus, without loss of generality, we can assume that
$j-j_1<j_2-(j+1)$, and the color $4$ appears in the first row of
the column $C_{j_1}$. Now two vertices $v_{3,j}$ and $v_{1,j+1}$
have the same distance $j-j_1+1$ to the color class $4$ and hence
have the same color codes, which is a contradiction. For $m\geq
4$, if two consecutive columns have the same missing colors, then
by an argument similar to the above, one can find two
vertices with the same color codes, which is impossible.\\ Now let
$m\geq 5$ and suppose on the contrary that there exist two (non
consecutive) columns $C_j$ and $C_{j'}$ with the same missing
color, say $``1"$. We know that there are no two consecutive
columns with the same missing colors and hence each of these two
columns contains at least one and at most two full colors.
Therefore, since $m\geq 5$, there are two vertices $v$ and $v'$ of the same color in
the $j$-th and $j'$-th columns of $G$, respectively, such that $v$ and $v'$ are not adjacent to a vertex colored 1. Thus, $v$ and $v'$ have the same color codes, which is a contradiction.
 }\end{proof}
\begin{thm}
Let  $m\geq 3$ and $n\geq2$ be two positive integers. Then
\begin{eqnarray*}
\Cchi_{_L}(K_m\square P_n)= \left\{
\begin{array}{ll}
m+2 & ~~{\rm if}~m\leq n-2, \\
m+1 & ~~{\rm if}~m\geq n-1.
\end{array}
\right.
\end{eqnarray*}
\end{thm}
\begin{proof}{
Let $G:=K_m\square P_n$. By Observation 1(b), we have $\Cchi_{_L}(G)\geq m+1$.\\
Now we give a locating $(m+2)$-coloring of $G$. Let the first
column of the corresponding coloring matrix be the column vector
$[(m+1)~1~2~3~...~(m-2)~(m+2)]^T$, and the remaining columns be
alternately $[1~2~3~...~m]^T$ and $[m~1~2~3~...~(m-1)]^T$. Then,
no two distinct vertices with the same color have the same
distances to both of the color classes $m+1$ and $m+2$. Hence,
this is a locating coloring of $G$. Therefore, $\Cchi_{_L}(G)=m+1$
or $\Cchi_{_L}(G)=m+2$.\\ First we show that if
$\Cchi_{_L}(G)=m+1$, then $m\geq n-1$.
Therefore, if $m\leq n-2$, then $\Cchi_{_L}(G)=m+2$.\\
Assume that $\Cchi_{_L}(G)=m+1$ and let $C$ be the corresponding
matrix of a locating $(m+1)$-coloring of $G$. Note that $C$ has
$m$ rows and in each column exactly one color is missing. By Lemma
\ref{ConsecutiveColumns}, no two consecutive columns of $C$ have
the same missing color, and hence each column of $C$ contains at
least one full color.
This implies that $G$ has at most $m+1$ columns, i.e. $n\leq m+1$ as desired.\\
To complete the proof, we assume $m\geq n-1$ and show that
$\Cchi_{_L}(G)=m+1$. For $m\in \{3,4\}$, consider two colorings of
$K_3\square P_4$ and $K_4\square P_5$ with the corresponding
matrices $A_1$ and $A_2$, respectively, as follows.
\begin{eqnarray*}
A_1=\left [
\begin{array}{llll}
1 & 4 & 2 & 3 \\
2 & 1 & 4 & 1 \\
3 & 2 & 3 & 4
\end{array}
\right] , ~~A_2=\left [
\begin{array}{lllll}
1 & 5 & 1 & 5 & 4 \\
2 & 3 & 5 & 3 & 5 \\
3 & 1 & 2 & 4 & 2 \\
4 & 2 & 4 & 1 & 3
\end{array}
\right] .
\end{eqnarray*}
Note that in these colorings distinct columns have distinct
missing colors and hence, two vertices with the same color have
distinct color codes except when both of them are colorful. There
are exactly $m+1$ colorful vertices (with distinct colors).  Thus,
these colorings are locating. Also note that removing columns from
the end will not create new full colors in the remaining matrices.
Thus, for $m\in \{3,4\}$ and $n\leq m+1$, we have
$\Cchi_{_L}(K_m\square P_n)=m+1$.\\
Now let $m\geq 5$. By Lemma \ref{ConsecutiveColumns}, in the
corresponding matrix of each locating $(m+1)$-coloring of $G$, if
it exists, there are no two distinct columns with the same missing
color. In an inductive way, we give a locating $(m+1)$-coloring of
$G$. Equivalently, we fill the columns of an $m$ by $n$ matrix $C$
with entries in $[m+1]$ in such a way that each column contains
exactly one full color, distinct columns have distinct full
colors, and the missing colors of no two columns are the same.
These coloring will be locating since there are no two colorful
vertices with the same color and, two non-colorful vertices with
the same color are in different columns and are non-adjacent to
different color classes. We construct this coloring matrix for
$n=m+1$, then for
smaller $n$ one can remove extra columns from the end.\\
Let $C_1:=[1~2~3~...~m]^T$ and $C_2:=[(m+1)~1~2~3~...~(m-1)]^T$ be
the first and second columns of $C$, respectively. Now assume that
$p$-th column of $C$ is $C_p=[x_1~x_2~x_3~...~x_m]^T$ with the
missing color $x_{m+1}$ and, without loss of generality, with the full
color $x_1$, where
$\{x_1,x_2,x_3,...,x_{m+1}\}=\{1,2,3,...,m+1\}$. Next, we fill the
column $C_{p+1}$. Since $C_p$ should contain exactly one full
color, the color $x_{m+1}$ should appear in the
first row of $C_{p+1}$. \\
For each $t\geq 1$, let $C_t^F$ and $C_t^M$ be the singleton sets that
contain the full color and the missing color of the column $C_t$,
respectively. If $x_1$ is not the missing color of one of the
previous columns, and $x_{m+1}$ is not the full color of one of
the previous columns, then let
$C_{p+1}:=[x_{m+1}~x_m~x_2~x_3~x_4~...~x_{m-1}]^T$, which means
$C_{p+1}^M=\{x_1\}$ and $C_{p+1}^F=\{x_{m+1}\}$. Otherwise,
$x_1\in \bigcup_{t=1}^{p}C_t^M$ or $x_{m+1}\in
\bigcup_{t=1}^{p}C_t^F$. If $p+1<n$, then there are at least two
colors not in $\bigcup_{t=1}^{p}C_t^F$ and there are at least two
colors not in $\bigcup_{t=1}^{p}C_t^M$. Therefore, there exists a
color $x_j\notin \bigcup_{t=1}^{p}C_t^F$, where $x_j\neq x_{m+1}$.
Also, there exists a color $x_i\notin \bigcup_{t=1}^{p}C_t^M$,
where $x_i\notin \{x_1,x_j\}$. Choose $x_j$ as the full color, and
$x_i$ as the missing color of $C_{p+1}$. Since $m\geq 5$, it is
possible to fill the column $C_{p+1}$ in such a way that its first
row is $x_{m+1}$ and its $i$-th row is $x_j$. Here after assume
that $p+1=n$, and hence there is only one color not in
$\bigcup_{t=1}^{p}C_t^F$ and only one color
not in $\bigcup_{t=1}^{p}C_t^M$. Following two cases may occur.\\

\noindent \textbf{Case 1. } $[n]\setminus
\bigcup_{t=1}^{n-1}C_t^M=\{x_1\}$ and
$[n]\setminus\bigcup_{t=1}^{n-1}C_t^F=\{x_i\}$, where $x_i\neq
x_{m+1}$.

\medskip \vspace*{1mm}
\noindent In this case, we could change some of the previous
columns to get the desired coloring. Assume that
$C_{n-2}^F=\{x_s\}$ and $C_{n-2}^M=\{x_j\}$. Note that there are
no repeated full colors or repeated missing colors, but it may be that
$x_s=x_{m+1}$ or $x_i=x_j$.
     \begin{itemize}
     \item[(a)] If $x_s\neq x_{m+1}$ and $x_i=x_j$, then let $x_j,
     x_1$ and $x_{m+1}$ be the missing colors of $C_{n-2}, C_{n-1}$ and
     $C_n$,
     respectively. Also let $x_1,
     x_j$ and $x_{s}$ be the full colors of $C_{n-2}, C_{n-1}$ and
     $C_n$,
     respectively. Now fill $C_{n-2}$ such that $x_1$ in $C_{n-2}$ and $x_j$ in
     $C_{n-3}$ are in the same row, and then fill $C_{n-1}$ such
     that $x_i=x_j$ in $C_{n-1}$ and $x_1$ in $C_{n-2}$ are in the
     same row. Finally, fill $C_{n}$ such
     that $x_1$ in $C_{n}$ and $x_i$ in $C_{n-1}$ are in the
     same row, and also $x_s$ in $C_{n}$ and $x_{m+1}$ in $C_{n-1}$ are in the
     same row.
     \item[(b)] If $x_s=x_{m+1}$ and $x_i\neq x_j$, then let $C_{n-2}^M=\{x_1\},
     C_{n-1}^M=\{x_s\}, C_n^M=\{x_j\}$ and $C_{n-2}^F=\{x_s\},
     C_{n-1}^F=\{x_1\}, C_n^F=\{x_i\}$.
     Now fill $C_{n-2}$ such that $x_1$ in $C_{n-2}$ and $x_s$ in
     $C_{n-3}$ are in the same row. Then fill $C_{n-1}$ such
     that $x_s$ in $C_{n-1}$ and $x_1$ in $C_{n-2}$ are in the
     same row. Finally, fill $C_{n}$ such
     that $x_1$ in $C_{n}$ and $x_s$ in $C_{n-1}$ are in the
     same row, and also $x_i$ in $C_{n}$ and $x_j$ in $C_{n-1}$ are in the
     same row..
     \item[(c)] If $x_s\neq x_{m+1}$ and $x_i\neq x_j$, then consider the following two cases.
     \\ If $C_{n-3}^F\neq \{x_j\}$, then let
     $$C_{n-2}^M=\{x_j\},~C_{n-1}^M=\{x_{m+1}\},~C_n^M=\{x_1\},$$
     $$C_{n-2}^F=\{x_1\},~C_{n-1}^F=\{x_s\},~C_n^F=\{x_i\}.$$
     If $C_{n-3}^F=\{x_j\}$, then (by assuming $C_{n-3}^M=\{x_l\}$) let
     $$C_{n-3}^M=\{x_l\},~C_{n-2}^M=\{x_{m+1}\},~C_{n-1}^M=\{x_j\},~C_n^M=\{x_1\},$$
     $$C_{n-3}^F=\{x_j\},~C_{n-2}^F=\{x_1\},~C_{n-1}^F=\{x_i\},~C_n^F=\{x_s\}.$$
     \item[(d)] If $x_s=x_{m+1}$ and $x_i=x_j$, then we
     should change the column $C_{n-3}$. Assume that $C_{n-3}^F=\{x_l\}$ and $C_{n-3}^M=\{x_k\}$.
     Note that $x_l\notin \{x_1,x_i,x_{m+1}\}$ and $x_k\notin \{x_1,x_i,x_{m+1}\}$,
     since there are no repeated full colors or repeated missing
     colors. For the desired coloring, let 
     $$C_{n-3}^M=\{x_1\},~C_{n-2}^M=\{x_k\},~
     C_{n-1}^M=\{x_i\},~ C_n^M=\{x_{m+1}\},$$
     $$C_{n-3}^F=\{x_i\},~ C_{n-2}^F=\{x_{m+1}\},~
     C_{n-1}^F=\{x_1\},~ C_n^F=\{x_l\}.$$
     \end{itemize}

\noindent \textbf{Case 2. } $[n]\setminus
\bigcup_{t=1}^{n-1}C_t^M=\{x_i\}$ and
$[n]\setminus\bigcup_{t=1}^{n-1}C_t^F=\{x_{m+1}\}$, where $x_i\neq
x_1$.

\medskip
\noindent We should change $C_{n-2}$ to get the desired coloring.
Assume that $C_{n-2}^F=\{x_s\}$ and $C_{n-2}^M=\{x_j\}$. Note that
$x_j\notin \{x_1,x_i,x_{m+1}\}$ and $x_s\notin \{x_1,x_{m+1}\}$,
     since there are no repeated full colors or repeated missing
     colors. But it may be that $x_i=x_s$.
     \begin{itemize}
     \item[(a)] If $x_i\neq x_s$, then consider the following two cases.
     \\ If $C_{n-3}^M\neq\{x_1\}$, then let
     $$C_{n-2}^M=\{x_{m+1}\},~C_{n-1}^M=\{x_i\},~C_n^M=\{x_j\},$$
     $$C_{n-2}^F=\{x_1\},~C_{n-1}^F=\{x_s\},~C_n^F=\{x_{m+1}\}.$$
     Now assume that $C_{n-3}^M=\{x_1\}$ and $C_{n-3}^F=\{x_l\}$.
     It is possible that $x_l\in\{x_i,x_j\}$. Without loss of generality, assume that $x_l\neq x_i$
     (when $x_l=x_i$, we replace $x_i$ by $x_j$). Let
     $$C_{n-3}^M=\{x_{m+1}\},~C_{n-2}^M=\{x_i\}
     ,~C_{n-1}^M=\{x_1\},~C_n^M=\{x_j\},$$
     $$C_{n-3}^F=\{x_1\},~C_{n-2}^F=\{x_l\},
     ~C_{n-1}^F=\{x_s\},~C_n^F=\{x_{m+1}\}.$$
     \item[(b)] If $x_i=x_s$, then let 
     $$C_{n-2}^M=\{x_{m+1}\},~C_{n-1}^M=\{x_s\},~ C_n^M=\{x_j\},$$
     $$C_{n-2}^F=\{x_s\},~C_{n-1}^F=\{x_{m+1}\},~ C_n^F=\{x_1\}.$$
     \end{itemize}
Note that since $m\geq 5$, in all of the previous steps it is possible to fill each column in the desired manner.
 }\end{proof}

\section{The locating chromatic number of $K_m\square K_n$}

In this section, we study the cartesian product of complete
graphs. Let $G:=K_m\square K_n$. Vertices of $G$ can be considered
as the entries of an $m$ by $n$ matrix, such that the induced
subgraph on the vertices of each column is isomorphic to $K_m$ and
the induced subgraph on each row is isomorphic to $K_n$. Let
$v_{i,j}$ be the vertex of $G$ in the $i$-th row and $j$-th
column. Each coloring of $G$ can also be considered as an $m$ by
$n$ matrix.
\begin{lemma}\label{n+1}
Let $m\geq 2$ and $n\geq 3$ be two positive integers, where $m\leq n$.
If there exists a locating $(n+1)$-coloring of $G:=K_m\square K_n$, then let $C$ be its
corresponding coloring matrix. Then different rows of
$C$ have different missing colors.
\end{lemma}
\begin{proof}{
Each row has one missing color. Since each color appears in at
least one row, the missing color of each row appears in some other rows. Hence, each row contains some full colors.\\
Suppose on the contrary, and without loss of generality, that
first and second rows have the same missing color, say ``$n+1$".
For each $i\in [n]$, there are two vertices in the first and
second rows of $G$ with color $i$. They have neighbors in all of
the color classes $[n]\setminus \{i\}$. Since the coloring is
locating, the color $n+1$ should appear in exactly one of the
columns corresponding to these two vertices. This holds for each
$i\in [n]$. Hence, the color $n+1$ should appear in exactly half
of the columns of $C$. This also implies that $n$ is an even
integer. Thus, in each row with the missing color $n+1$, half of the
colors are full. Particularly, half of the colors $1, 2, ..., n$
are full in the first row, and the remaining are full in the
second row. Since repeated full colors are not allowed, the color
$n+1$ must appear in the third row. The missing color of the third row appears in the first and second rows, in two different columns. Hence, the third row
contains at least two full colors. This implies that
there are at least $n+2$ full colors in $n+1$ color classes,
which is a contradiction.
 }\end{proof}
Note that $\Cchi_{_L}(K_2\square K_2)=\Cchi_{_L}(C_4)=4$. In
general we have the following result.
\begin{thm}
For two positive integers $m\geq 2$ and $n\geq 3$, where $m\leq
n$, let
$$m_0:=\max\{k~|~k\in \Bbb{N},~k(k-1)-1\leq n\}.$$
   \begin{itemize}
   \item [{\rm (a)}] If $m\leq m_0-1$, then $\Cchi_{_L}(K_m\square K_n)=n+1$,
   \item [{\rm (b)}] If $m_0+1\leq m\leq {n\over 2}$, then $\Cchi_{_L}(K_m\square K_n)=n+2$.
   \end{itemize}
\end{thm}
\begin{proof}{
Let $G:=K_m\square K_n$. By Observation 1(b), we have $\Cchi_{_L}(G)\geq n+1$.\\
If $m=2$, then the following matrix provides a locating
$(n+1)$-coloring of $G$ with the color set $[n+1]$.
\begin{eqnarray*}
\left [
\begin{array}{ccccc}
1 & 2 & 3 & ... & n \\
n+1 & 1 & 2 & ... & n-1
\end{array}
\right]
\end{eqnarray*}
Let $m=3$. If $n=3$, then it is not hard to see that there exists
no locating $4$-coloring of $K_3\square K_3$.
The following matrix $A_1$ gives a locating $5$-coloring of $K_3\square K_3$. Hence $\Cchi_{_L}(K_3\square K_3)=5$.\\
If $n=4$ and $\Cchi_{_L}(K_3\square K_4)=n+1=5$, then by Lemma
\ref{n+1} different rows have different missing colors. Hence,
each row contains two full colors, which is impossible since there
are only five color classes. The following matrix $A_2$ gives a
locating $6$-coloring of $K_3\square K_4$, and hence
$\Cchi_{_L}(K_3\square K_4)=6$. If $n=5$, then the matrix $A_3$
gives a locating $6$-coloring  of $K_3\square K_5$, and hence
$\Cchi_{_L}(K_3\square K_5)=6$.

\begin{eqnarray*}
A_1=\left [
\begin{array}{lll}
1 & 2 & 3 \\
4 & 1 & 2 \\
2 & 5 & 4
\end{array}
\right] ,~~ A_2=\left [
\begin{array}{llll}
1 & 2 & 3 & 4 \\
5 & 1 & 2 & 3 \\
6 & 5 & 1 & 2
\end{array}
\right] ,~~ A_3=\left [
\begin{array}{lllll}
1 & 5 & 3 & 4 & 2 \\
6 & 1 & 5 & 2 & 4 \\
3 & 4 & 2 & 5 & 6
\end{array}
\right]
\end{eqnarray*}
Finally, the following matrix gives a locating $(n+1)$-coloring of
$n\geq 6$, and hence $\Cchi_{_L}(K_3\square K_n)=n+1$.
\begin{eqnarray*}
\left [
\begin{array}{ccccccc}
1 & 2 & 3 & ... & n-2 & n-1 & n \\
n+1 & 1 & 2 & ... & n-3 & n-2 & n-1 \\
2 & 3 & 4 & ... & n-1 & n & n+1
\end{array}
\right]
\end{eqnarray*}
Here after let $m\geq 4$. First assume there exists a locating
$(n+1)$-coloring $c$ of $G$ with the corresponding coloring matrix
$C$. By Lemma \ref{n+1}, missing colors of different rows are
different, and the missing color of each row appears in all of the
other $m-1$ rows. Hence, each row of $C$ contains exactly $m-1$
full colors, and there  are exactly $m(m-1)$ full colors in $C$.
Since there are $n+1$ color classes, we should have $m(m-1)\leq
n+1$. Therefore, if $m>m_0$, then
$\Cchi_{_L}(G)\geq n+2$.\\
Now we provide a locating $(n+2)$-coloring of $G$, when $m_0+1\leq
m\leq {n\over 2}$. Let $c:V(G)\longrightarrow [n+2]$ be a
function, where $c(v_{i,j}):=(i-1)n+j$ ~(mod $(n+2)$). Since
$m\leq {n\over 2}$, by a simple calculation, it can be seen that
$c$ is a proper coloring. For each $i$, $1\leq i\leq m$, the
$i$-th row of the corresponding coloring matrix has two missing
colors $in+1$ and $in+2$, modulo $n+2$. Moreover, since $m\leq
{n\over 2}$, two colors $in+1$ and $in+2$ can not appear in the
same column. This means that there exists no full color. In
other words, each vertex has at least one component that is 2 in its
color code. Since each color is missed in exactly one row, every
two vertices with the same color are in different rows and have
different color codes. Consequently, $c$ is a locating
$(n+2)$-coloring.
Hence, $\Cchi_{_L}(G)=n+2$, when $m_0<m\leq {n\over 2}.$\\
To complete the proof, we show that $\Cchi_{_L}(G)=n+1$, when
$m<m_0$. In an inductive way, we provide a locating
$(n+1)$-coloring of $K_m\square K_n$. For this purpose, we
construct an $m$ by $n$ (coloring) matrix on the set $[n+1]$ with
the following properties.
\begin{itemize}
\item [(a)] The entries of each row, and each column are different.
\item [(b)] Different rows have different missing colors.
\item [(c)] There exist no repeated full colors.
\item [(d)] All of the missing colors are also full.
\end{itemize}
Note that the property (a) indicates that the given coloring is
proper. Also, the properties (b) and (c) guarantee that the coloring
is a locating coloring, since two non-colorful vertices with the
same color are in different rows and are non-adjacent to different
color classes. The property (d) is needed for the proof by
induction. If $m=4$, then $m_0\geq 5$ and $n\geq 19$. Consider the
following coloring matrix.
\begin{eqnarray*}
\left [
\begin{array}{cccccccccccccc}
1 & 2 & 3 & 4 & 5 & ... & i & ... & n-5 & n-4 & n-3 & n-2 & n-1 & n \\
n+1 & 1 & 2 & 3 & 4 & ... & i-1 & ... & n-6 & n-5 & n-4 & n-3 & n-2 & n-1 \\
n & n+1 & 1 & 2 & 3 & ... & i-2 & ... & n-7 & n-6 & n-5 & n-4 & n-3 & n-2 \\
4 & 5 & 6 & 7 & 8 & ... & i+3 & ... & n-2 & n-1 & n & 2 & 3 & n+1
\end{array}
\right]
\end{eqnarray*}
In this coloring, the missing colors are $n+1, n, n-1, 1$, and the
full colors are $1, 2, n; n+1, n-4, n-1; n-6, n-3, n-2; 4, 5, 6$.
Thus, there are no repeated full colors, whenever $n\geq 13$, and
properties (a) to (d) hold.\\
Now suppose that the $i$-th row, $4<i\leq m-1$, is completed
such that the properties (a) to (d) hold for the constructed
$i\times n$ matrix. Next, we complete the $(i+1)$-th row.
Without loss of generality and by permuting the rows or symbols if
it is necessary, assume that the missing colors in the first $i$
rows are $1, 2, ..., i$. Each of the first $i$ rows contains $i-1$
full colors, since its missing color appears in other rows. Thus,
there are $i(i-1)$ full colors and $i$ missing colors. Choose a
full color $j$, $j>i$. We want to fill the $(i+1)$-th row with
colors in the set $[n+1]\setminus \{j\}$ in such a way that the
constructed $(i+1)\times n$ matrix satisfies the properties (a) to
(d). \\
By completing the $(i+1)$-th row, $2i$ new colorful vertices
appear, $i$ of them will be in the first $i$ rows (one vertex in
each row) by inserting the colors $1, 2, ..., i$ in the
$(i+1)$-th row, and $i$ of them will be in the $(i+1)$-th row
corresponding to the columns in which $j$ occurs in the previous rows.\\
Let $1\leq k\leq i$. The color $k$ should be inserted in a
suitable column in the $(i+1)$-th row, in such a way that it
creates a new full color in the $k$-th row. Since the colors
$1, 2, ..., i$ are full, to preserve the property (c), these
colors shouldn't be inserted in the columns of the $(i+1)$-th
row in which $j$ occurs in the previous rows. There are $n-i$ columns
not containing $j$. On the other hand, $i(i-1)$ full colors of
the first $i$ rows  appear in the $k$-th row and inserting $k$
in their columns causes repeated full colors. Also, one of these
full colors is $j$ and each column containing $k$ contains at
least one full color. Thus, there are at least
$(n-i)-i(i-1)+1=n-i^2+1$ possible columns in the $(i+1)$-th row
for inserting the color $k$. Assume that the color $1$ is
inserted in a suitable column. After inserting $1$, one new full
color is created in the first row and one of the feasible columns
of the $(i+1)$-th row is occupied by $1$. Hence, for inserting
$2$ there are at least $(n-i^2+1)-2$ possible columns, and finally
for inserting $i$ there are at least $(n-i^2+1)-2(i-1)$ possible
columns. Note that $(n-i^2+1)-2(i-1)\geq 1$, since $i<m<m_0$ and
$m_0(m_0-1)-1\leq n$. Thus, inserting the colors $1, 2, ..., i$
is possible as desired.\\ Inserting each color in the columns in which $j$ occurs in the previous rows, will make that color full. There are $i$ columns
containing the color $j$. Since
\begin{eqnarray*}
(n+1)-i(i-1)-i&=& n+1-i^2 \\
&\geq&m_0(m_0-1)-i^2\\
&\geq& m_0(m_0-1)-(m_0-2)^2\\
&=&3(m_0-2)+2\\
&\geq& 3i+2,
\end{eqnarray*}
there are at least $3i+2$ non full colors. Therefore, it is
possible to insert $i$ non full colors in the $(i+1)$-th
row, and in the columns in which $j$ occurs in the first $i$ rows,
preserving the property (a).\\ Now it remains to insert the
remaining $n-2i$ colors, say $c_1, c_2,..., c_{n-2i}$, in the
remaining $n-2i$ columns, say $C_1, C_2,..., C_{n-2i}$, preserving
the property (a). Let $H:=(X,Y)$ be the bipartite graph with
partite sets $X:=\{C_1, C_2,..., C_{n-2i}\}$ and $Y:=\{c_1,
c_2,..., c_{n-2i}\}$ such that $C_sc_{_r}\in E(H)$, whenever the
color $c_r$ is not occurred in the column $C_s$. Each color $c_r$
is in $i$ rows and each column $C_s$ contains $i$ colors. Thus,
each vertex in $H$ has degree at least $n-3i$. Let $\emptyset
\subset S \subset X$. Since $S\neq\emptyset$, $|N(S)|\geq n-3i$.
If $|N(S)|<|S|$, then $n-3i<|S|$. Thus, $N(S)\neq Y$ and
$|X\setminus S|\leq i-1$. Let $y\in Y\setminus N(S)$ and hence,
$$n-3i\leq |N(y)|\leq |X\setminus S|\leq i-1.$$ Thus, $n\leq 4i-1$
and  $$m_0(m_0-1)-1\leq 4i-1\leq 4(m_0-1)-1.$$ This implies that
$m_0\leq 4$ which is a contradiction, since $4\leq m<m_0$.
Therefore, the Hall's condition holds (Theorem 3.1.11 \cite{West})
and hence $H$ has a perfect matching. Consequently, we obtain a
desired coloring by filling the remaining entries according to
this assignment.
 }\end{proof}

\section{Some open problems}
Note that every proper coloring of $K_m\square K_n$ is equivalent
to an $m$ by $n$ Latin rectangle. Moreover, a locating coloring of
$K_m\square K_n$ is equivalent to an $m$ by $n$ Latin rectangle
in which, for every two cells containing the same symbol, there is
a symbol that appears only in the row or column of one of them.\\
In what follows we present some open problems related to the obtained
results.\\\\
Note that for each given number $n$,
$m_0=\max\{k~|~k\in\Bbb{N},~k(k-1)-1\leq n\}$ is a number close to
$\sqrt n$. If $m=m_0\geq 4$ and $(m_0(m_0-1)-1)+(m_0-2)\leq n<
(m_0+1)m_0-1$, then for each $i$ with $i<m_0$, we have
$$(n-i^2+1)-2(i-1)\geq 1,~~ (n+1)-i(i-1)-i\geq 2m_0-3,$$ and $n\nleqslant 4i-1$.
Thus, by following the proof of the Theorem 3, we can obtain a
locating $(n+1)$-coloring of $K_{m_0}\square K_n$ and hence,
$\Cchi_{_L}(K_{m_0}\square K_n)= n+1$. Therefore, if $m=m_0$, then
the remaining cases for $n$ to investigate
$\Cchi_{{}_L}(K_{m_0}\square K_n)$ are
$$m_0(m_0-1)-1\leq n \leq (m_0(m_0-1)-1)+(m_0-3).$$ Verifying
small cases, encourage us to give the following conjecture.
\begin{conj}
$\Cchi_{_L}(K_{m_0}\square K_n)= n+1$, where $n\geq 3$ and
$m_0=\max\{k~|~k\in \Bbb{N},~k(k-1)-1\leq n\}.$
\end{conj}
By a long detailed argument, we can prove that
$\Cchi_{{}_L}(K_m\square K_n)=n+2$ for ${n\over 2}<m\leq
\lceil{n+3\over 2}\rceil$. For the remaining cases we provide the
following conjecture.
\begin{conj}
 If $\lceil
{n+3\over 2}\rceil <m\leq n$, then $\Cchi_{_L}(K_m\square
K_n)=n+3.$
\end{conj}
The similarity of the structures of $K_m\square C_n$ and
$K_m\square P_n$ is a motivation for the following conjecture.
\begin{conj}
If $m$ and $n$ are sufficiently large, then $\Cchi_{_L}(K_m\square
C_n)=\Cchi_{_L}(K_m\square P_n)$.
\end{conj}
It seems that graphs with bigger diameter have smaller locating
chromatic number. Hence, the obtained results suggest the following
conjecture.
\begin{conj}
For every two connected graphs $G$ and $H$, $\Cchi_{_L}(G\square
H)\leq \max\{\Cchi_{_L}(G),\Cchi_{_L}(H)\}+3.$
\end{conj}



\begin{thebibliography}{99}


\bibitem{Asmiati}{\sc Asmiati, H. Assiyatun,} and {\sc E.T. Baskoro}, Locating-chromatic number of amalgamation of stars,
{\it ITB J. Sci.} {\bf 43, A} (2011) 1-8.
\bibitem{Behtoei}{\sc A. Behtoei} and  {\sc B. Omoomi}, On the locating chromatic number of Kneser graphs,
{\it to appear in Discrete  Appl. Math.}
\bibitem{XL}{\sc G. Chartrand, D. Erwin, M.A. Henning, P.J. Slater,} and {\sc P. Zhang}, The locating-chromatic number of a graph,
{\it Bull. Inst. Combin. Appl.} {\bf 36} (2002) 89-101.
\bibitem{XLn-1}{\sc G. Chartrand, D. Erwin, M.A. Henning, P.J. Slater,} and {\sc P. Zhang}, Graphs of order $n$ with locating-chromatic number $n-1$,
{\it Discrete Math.} {\bf no. 1–3, 269} (2003)  65-79.
\bibitem{Metric09}{\sc G. Chartrand, F. Okamoto,} and {\sc P. Zhang}, The metric chromatic number of a graph,
{\it Australasian Journal of Combinatorics} {\bf 44 } (2009)
273-286.
\bibitem{ResolvingEdg}{\sc G. Chartrand, V. Saenpholphat,} and {\sc P. Zhang},  Resolving edge colorings in graphs, {\it  Ars Combin.} {\bf 74} (2005) 33-47.
\bibitem{partition}{\sc G. Chartrand, E. Salehi,} and {\sc P. Zhang}, The partition dimension of a graph, {\it Aequationes Math.} {\bf no. 1-2, 59} (2000)  45-54.
\bibitem{Harary}{\sc F. Harary,} and  {\sc R.A. Melter}, On the metric dimension of a graph, {\it Ars Combin.} {\bf 2} (1976) 191-195.
\bibitem{Conditional}{\sc V. Saenpholphat} and {\sc P. Zhang}, Conditional resolvability in graphs: A survey, {\it Int. J. Math. Sci.} {\bf 37-40} (2004) 1997-2017.
\bibitem{West}{\sc D. B. West}, Introduction to graph theory, {\it Prentice Hall Inc., Upper Saddle River, NJ, second edition,} (2001).

\end{thebibliography}
\end{document}